\documentclass[12pt,reqno]{amsart}

\usepackage{amssymb}

\newtheorem{theorem}{Theorem}[section]
\newtheorem{proposition}[theorem]{Proposition}
\newtheorem{lemma}[theorem]{Lemma}

\theoremstyle{definition}
\newtheorem{definition}[theorem]{Definition}

\numberwithin{equation}{section}

\begin{document}

\baselineskip=15.5pt

\title[Logarithmic Cartan geometry]{Logarithmic Cartan geometry on complex manifolds}

\author[I. Biswas]{Indranil Biswas}

\address{School of Mathematics, Tata Institute of Fundamental
Research, Homi Bhabha Road, Mumbai 400005, India}

\email{indranil@math.tifr.res.in}

\author[S. Dumitrescu]{Sorin Dumitrescu}

\address{Universit\'e C\^ote d'Azur, CNRS, LJAD}

\email{dumitres@unice.fr}

\author[B. McKay]{Benjamin McKay}

\address{University College Cork, Cork, Ireland}

\email{b.mckay@ucc.ie}



\date{}

\begin{abstract}
We pursue the study of holomorphic Cartan geometry with singularities. We introduce the notion of 
logarithmic Cartan geometry on a complex manifold, with polar part supported on a normal crossing 
divisor. In particular, we show that the push-forward of a Cartan geometry constructed using a 
finite Galois ramified covering is a logarithmic Cartan geometry (the polar part is supported on
the ramification locus). We also study the specific case of the logarithmic Cartan geometry with 
the model being the complex affine space.
\end{abstract}

\maketitle

\tableofcontents

\section{Introduction}

In a vast generalization of Riemannian geometry, \'E.~Cartan introduced and studied Cartan 
geometries (or Cartan connections) which are geometric structures infinitesimally modelled on 
homogeneous spaces (see, for example, the excellent survey \cite{Sh}). In particular, Cartan's 
theory encapsulates the study of affine and projective connections on manifolds. It may be 
recalled that, historically, the study of complex projective structures (i.e., (flat) Cartan 
geometries modelled on the complex projective line) on Riemann surfaces had played a crucial 
r\^ole in the understanding of the uniformization theorem for Riemann surfaces \cite{Gu, St}.

In higher dimension, it is a very stringent condition for a compact complex manifold to admit a 
holomorphic Cartan geometry. In this direction, several authors proved classifications results 
for compact complex manifolds bearing holomorphic Cartan geometries (see, for example, \cite{BD1, 
BD2, BM, Du, IKO, JR, KO1, KO2, KO3, PR}).

The notion of a (nonsingular) holomorphic Cartan geometry on a compact complex manifold being too 
rigid, it seems natural to allow mild singularities of the geometric structure. In this direction 
the first two authors introduced and studied in \cite{BD1} the more flexible concept of {\it 
branched Cartan geometry} which is stable by pull-back through any holomorphic ramified map (see 
also \cite{BD2}). In particular, any compact complex projective manifold admits (flat) branched 
complex projective structures (locally modelled on the complex projective space of the same 
dimension) \cite{BD1, BD2}.

We pursue here the study of Cartan geometries with singularities and the aim of this article is 
to introduce the notion of {\it logarithmic Cartan geometry.} To explain with more details, we 
define logarithmic Cartan geometries (on complex manifolds) with model $(G,\,H)$, where $G$ is a 
complex affine Lie group (e.g. admitting linear holomorphic representations with discrete kernel) 
and $H$ is a closed complex subgroup in it. On the complement of the support of the singular 
(polar) part (which is allowed to be a normal crossing divisor) we recover the classical 
definition of a holomorphic Cartan geometry with model $(G,\,H)$. The extension of the Cartan 
geometry across the polar part is realized by an extension of a linear bundle associated to the 
holomorphic principal $G$-bundle of the Cartan geometry through a linear representation (with 
discrete kernel) of the group $G$ together with an extension on it of the natural connection 
inherited by the Cartan geometry as a logarithmic connection. This is worked out with details in 
Section \ref{sect2} and Section \ref{sect3}. Our definition generalizes the notions of 
logarithmic affine and projective connections on complex manifolds introduced and studied by Kato 
in \cite{Ka}. In particular, \cite{Ka} constructs interesting examples of compact complex simply 
connected non-K\"ahler manifolds admitting logarithmic holomorphic projective connections, that 
admit no holomorphic projective connections (with empty singular part).
 
In Section \ref{sect3} we also prove Theorem \ref{thm1} which asserts that the push-forward 
of a holomorphic Cartan geometry through a finite Galois ramified cover is a logarithmic 
Cartan geometry in our sense. In this case the support of the polar part coincides with the 
ramification locus. It may be recalled that the related topics of Cartan geometries on 
orbifolds was studied in \cite{Zh}.
 
Section \ref{sect4} is focused on a specific study of the logarithmic Cartan geometry whose
model is the complex affine space.

\section{Logarithmic connection}\label{sect2}

Let $M$ be a connected complex manifold of complex dimension $d$.
The holomorphic tangent bundle of $M$ will be denoted by $TM$, while its holomorphic cotangent
bundle of it will be denoted by $\Omega^1_M$.

A reduced effective divisor
$D\, \subset\, M$ is said to be a \textit{normal crossing divisor} if for every point $x\, \in\, D$
there are holomorphic coordinate functions $z_1,\, \cdots ,\, z_d$ defined on an Euclidean open
neighborhood $U\, \subset\, M$ of $x$ with $z_1(x)\,= \, \cdots\, = \, z_d(x)\,=\, 0$,
and there is an integer $1\, \leq\, k\, \leq\, d$, such that
\begin{equation}\label{e1}
D\cap U\, =\, \{y\, \in\, U \, \mid\, z_1(y)\, \cdot \, \cdots \, \cdot \, z_k(y)\,=\, 0\}
\end{equation}
(cf. \cite{Co}). Note that it is not assumed here that the irreducible components of the
divisor $D$ are smooth.

Take a normal crossing divisor $D$ on $M$. Let
$$
TM(-\log D)\, \subset\, TM
$$
be the coherent analytic subsheaf generated by all locally defined holomorphic vector fields $v$
on $M$ such that $v({\mathcal O}_M(-D))\, \subset\, {\mathcal O}_M(-D)$. In other words, if $v$
is a holomorphic vector field defined over $U\, \subset\, M$, then
$v$ is a section of $TM(-\log D)\vert_U$ if and only if $v(f)\vert_{U\cap D}\,=\, 0$ for
all holomorphic functions $f$ on $U$ that vanish on $U\cap D$. It is straightforward to check that
the stalk of sections of $TM(-\log D)$ at the point $x$ in \eqref{e1} is generated by
$$
z_1\frac{\partial}{\partial z_1},\, \cdots,\, z_k\frac{\partial}{\partial z_k},\,
\frac{\partial}{\partial z_{k+1}} ,\, \cdots,\, \frac{\partial}{\partial z_d}\, .
$$
The condition that $D$ is a normal crossing divisor implies that the coherent analytic sheaf
$TM(-\log D)$ is in fact locally free. Note that we have $TM\otimes {\mathcal O}_M(-D)\, \subset\,
TM(-\log D)$; this inclusion is strict if $\dim M\, >\, 1$.

Restricting the above inclusion homomorphism $TM(-\log D)\, \hookrightarrow\, TM$ to the divisor $D$,
we obtain a homomorphism
\begin{equation}\label{fn}
\psi\, :\, TM(-\log D)\vert_D\, \longrightarrow\, TM\vert_D
\end{equation}
Let
\begin{equation}\label{f1}
{\mathbb L}\, :=\, \text{kernel}(\psi)\, \subset\, TM(-\log D)\vert_D
\end{equation}
be the kernel. To describe $\mathbb L$, let
$$
\nu\, :\, \widetilde{D}\, \longrightarrow\, D
$$
be the normalization of the divisor $D$; the given condition on $D$ implies that
this $\widetilde{D}$ is smooth. Now $\mathbb L$ is identified with the direct image
\begin{equation}\label{f0}
{\mathbb L}\, =\, \nu_*{\mathcal O}_{\widetilde{D}}\, ,
\end{equation}
where $\nu$ is the above projection.
The key point in the construction of the isomorphism in \eqref{f0} is the following: Let $Y$ be a
Riemann surface and $y_0\, \in\, Y$ a point; then for any holomorphic coordinate function $z$ around 
$y_0$, with $z(y_0)\,=\,0$, the evaluation of the local section $z\frac{\partial}{\partial z}$ 
of $TY\otimes {\mathcal O}_Y(-y_0)$ at the point $y_0$ does not depend on the choice of the 
coordinate function $z$.

Consider the Lie bracket operation on the locally defined holomorphic vector fields on $M$. 
It can be shown that the holomorphic sections of $TM(-\log D)$ are closed under this Lie 
bracket operation. Indeed, if $v_1,\, v_2$ are holomorphic sections of $TM(-\log D)$ over 
$U\, \subset\, M$, and $f$ is a holomorphic function on $U$ that vanishes on $U\cap D$, then 
from the identity
$$
[v_1,\, v_2](f)\, =\, v_1(v_2(f)) - v_2(v_1(f))
$$
we conclude that the function $[v_1,\, v_2](f)$ also vanishes on $U\cap D$.

The dual vector bundle $TM(-\log D)^*$ is denoted by
$\Omega^1_M(\log D)$. Note that
$$
(TM)^* \,=\, \Omega^1_M\, \subset\, \Omega^1_M(\log D)\, ;
$$
the inclusion of $\Omega^1_M$ in $\Omega^1_M(\log D)$ is the dual of the
inclusion of $TM(-\log D)$ in $TM$.

For every integer $i \, \geq\, 0$, define $$\Omega^i_M(\log D)\,:=\, \bigwedge\nolimits^i
\Omega^1_M(\log D)\, .$$

Let
$$
\eta\, :\, D\, \hookrightarrow\, M
$$
be the inclusion map. Taking dual of the homomorphism $\psi$ (see \eqref{fn}), and using
\eqref{f0}, we get the following short exact sequence of coherent analytic sheaves on $M$
$$
0\, \longrightarrow\, \Omega^1_M\,\longrightarrow\, \Omega^1_M(\log D)\,
\stackrel{\mathcal R}{\longrightarrow}\, (\eta\circ\nu)_*{\mathcal O}_{\widetilde{D}}\, 
\longrightarrow\, 0
\, ,
$$
where $\nu$ is the map in \eqref{f0} and $\eta$ is the above inclusion
map of $D$; the above homomorphism $\mathcal R$
is known as the \textit{residue} map.

We refer the reader to \cite{Sa} for more details on logarithmic forms and logarithmic
vector fields.

Now let $H$ be a complex Lie group. The Lie algebra of $H$ will be denoted by $\mathfrak h$.
Let
\begin{equation}\label{e3}
p\, :\, E_H\, \longrightarrow\, M
\end{equation}
be a holomorphic principal $H$--bundle; we recall that this means that $E_H$ is a
holomorphic fiber bundle over $M$ equipped with a holomorphic right-action of the group $H$
\begin{equation}\label{e4}
q'\, :\, E_H\times H\,\longrightarrow\, E_H
\end{equation}
such that $p(q'(z,\, h))\,=\, p(z)$ for all $(z,\, h)\, \in\, E_H\times H$, where
$p$ is the projection in \eqref{e3} and, furthermore, the resulting
map to the fiber product
$$
E_H\times H \, \longrightarrow\, E_H\times_M E_H\, , \ \ (z,\, h) \, \longrightarrow\, (z,\,
q'(z,\, h))
$$
is a biholomorphism. For notational convenience, the point
$q'(z,\, h)\, \in\, E_H$, where $(z,\, h)\,\in\, E_H\times H$, will be denoted by $zh$.

Let $dp\, :\, TE_H\, \longrightarrow\, p^*TM$ be the differential of the projection $p$
in \eqref{e3}. Let
$$
{\mathcal K} \, :=\, \text{kernel}(dp)\, \subset\, TE_H
$$
be the kernel of $dp$. So we have the following short exact sequence of holomorphic vector bundles
on $E_H$:
\begin{equation}\label{f2}
0\, \longrightarrow\, {\mathcal K}\, \longrightarrow\, TE_H\, \stackrel{dp}{\longrightarrow}\,
p^*TM \, \longrightarrow\,0\, .
\end{equation}

Consider the action of $H$ on the tangent bundle $TE_H$
given by the action of $H$ on $E_H$ in \eqref{e4}. The quotient $(TE_H)/H$ is a holomorphic vector
bundle over $E_H/H\,=\, M$. It is the Atiyah bundle for $E_H$; let ${\rm At}(E_H)$ denote this
Atiyah bundle (see \cite{At}). 

The action of $H$ on $TE_H$ evidently preserves
the subbundle $\mathcal K$ in \eqref{f2}. The quotient
$$
\text{ad}(E_H)\,:=\, {\mathcal K}/H \,\, \longrightarrow\, E_H/H\,=\, M
$$
is called the \textit{adjoint vector bundle} for $E_H$. We note that $\text{ad}(E_H)$ is 
identified with the holomorphic vector bundle $E_H\times^H \mathfrak h\, \longrightarrow\, M$ 
associated to the principal $H$--bundle $E_H$ for the adjoint action of $H$ on the Lie algebra 
$\mathfrak h$. This isomorphism between ${\mathcal K}/H$ and $E_H\times^H \mathfrak h$ is 
obtained from the fact that the action of $H$ on $E_H$ identifies $\mathcal K$ with the trivial 
holomorphic vector bundle $E_H \times \mathfrak h$ over $E_H$ with fiber $\mathfrak h$. 
Therefore, every fiber of $\text{ad}(E_H)$ is a Lie algebra isomorphic to $\mathfrak h$.

Taking quotient of the vector bundles in \eqref{f2} by the actions of $H$, from \eqref{f2} we get 
a short exact sequence of holomorphic vector bundles over $M$
\begin{equation}\label{f3}
0\, \longrightarrow\, \text{ad}(E_H)\,:=\,{\mathcal K}/H\, \longrightarrow\,
(TE_H)/H \,=:\, {\rm At}(E_H) \, \stackrel{\beta'}{\longrightarrow}\,
(p^*TM)/H \,=\, TM \, \longrightarrow\,0\, ,
\end{equation}
which is known as the Atiyah exact sequence for $E_H$ (see \cite{At});
the differential $dp$ descends to the surjective homomorphism $\beta'$ in \eqref{f3}.

A \textit{holomorphic connection} on $E_H$ is a holomorphic homomorphism of vector bundles
$$
\varphi'\, :\, TM \, \longrightarrow\, {\rm At}(E_H)
$$
such that $\beta'\circ\varphi'\,=\, \text{Id}_{TM}$ \cite{At}.

As before, let $D\, \subset\, M$ be a normal crossing divisor. Since $p$ in \eqref{e3} is
a holomorphic submersion, the inverse image
$$
\widehat{D} \,:=\, p^{-1}(D)\, \subset\, E_H
$$
is also a normal crossing divisor. The action of $H$ on the tangent bundle $TE_H$,
given by the holomorphic action of $H$ on $E_H$ in \eqref{e4}, clearly
preserves the subsheaf $TE_H(-\log \widehat{D})\, \subset\, TE_H$. The corresponding quotient
$$
{\rm At}(E_H)(-\log D)\, :=\, TE_H(-\log \widehat{D})/H \, \longrightarrow\,M
$$
is evidently a holomorphic vector bundle over $M$; it is called the
\textit{logarithmic Atiyah bundle}.

Note that we have ${\mathcal K} \,\subset\, TE_H(-\log \widehat{D})$, and also
$dp(TE_H(-\log \widehat{D}))\,=\, p^*(TM(-\log D))$. Therefore, the
short exact sequence in \eqref{f2} gives the following
short exact sequence of holomorphic vector bundles over $E_H$
\begin{equation}\label{e7}
0\, \longrightarrow\, {\mathcal K}\, \longrightarrow\, TE_H(-\log \widehat{D})\,
\stackrel{d'p}{\longrightarrow}\, p^*(TM(-\log D))\, \longrightarrow\,0\, ;
\end{equation}
the restriction of the homomorphism
$dp$ in \eqref{f2} to $TE_H(-\log \widehat{D})$ is denoted by $d'p$.

Exactly as done in \eqref{f3}, take
quotient of the vector bundles in \eqref{e7} by the actions of $H$. From \eqref{e7} we get a
short exact sequence of holomorphic vector bundles over $M$
\begin{equation}\label{e8}
0\, \longrightarrow\, \text{ad}(E_H)\,:=\,{\mathcal K}/H\, \stackrel{\iota_0}{\longrightarrow}\,
(TE_H(-\log \widehat{D}))/H \,=:\, {\rm At}(E_H)(-\log D)
\end{equation}
$$
\stackrel{\beta}{\longrightarrow}\,
(p^*(TM(-\log D)))/H \,=\, TM(-\log D) \, \longrightarrow\,0\, ;
$$
it is called the \textit{logarithmic Atiyah exact sequence} for $E_H$. The homomorphism
$\beta$ in \eqref{e8} is the restriction $\beta'$ in \eqref{f3}.

A \textit{logarithmic connection} on $E_H$ singular over $D$ is a holomorphic homomorphism
of vector bundles
$$
\varphi\, :\, TM(-\log D)\, \longrightarrow\,{\rm At}(E_H)(-\log D)
$$
such that
\begin{equation}\label{e10}
\beta\circ\varphi\,=\, \text{Id}_{TM(-\log D)}\, ,
\end{equation}
where $\beta$ is the projection in \eqref{e8}. In other words, giving a logarithmic connection 
on $E_H$ singular over $D$ is equivalent to giving a holomorphic splitting of the short exact 
sequence in \eqref{e8}. See \cite{De} for logarithmic connections (see also \cite{BHH}).

\subsection{Curvature}

As noted before, the locally defined holomorphic sections of the
logarithmic tangent bundles $TM(-\log D)$ and $TE_H(-\log 
\widehat{D})$ are closed under the Lie bracket operation of vector fields. The locally defined 
holomorphic sections of the subbundle $\mathcal K$ in \eqref{f2} are clearly closed under the 
Lie bracket operation. The homomorphisms in the exact sequence \eqref{f2} are all compatible 
with the Lie bracket operation. Since the Lie bracket operation commutes with diffeomorphisms, 
for any two $H$--invariant holomorphic vector fields $v,\, w$ defined on an $H$--invariant 
open subset of $E_H$, their Lie bracket $[v,\, w]$ is again holomorphic and $H$--invariant. Therefore, 
the sheaves of sections of the three vector bundles in \eqref{e8} are all equipped with a Lie bracket 
operation. Moreover, all the homomorphisms in \eqref{e8} commute with these operations.

Take a homomorphism
$$
\varphi\, :\, TM(-\log D)\, \longrightarrow\,{\rm At}(E_H)(-\log D)
$$
satisfying the condition stated in \eqref{e10}. Then for any two holomorphic sections
$v_1,\, v_2$ of $TM(-\log D)$ over $U\, \subset\, M$, consider
$$
{\mathbb K}(v_1,\, v_2)\, :=\, [\varphi(v_1),\, \varphi(v_2)]- \varphi([v_1,\, v_2])\, .
$$
The projection $\beta$ in \eqref{e8} intertwines the Lie bracket operations
on the sheaves of sections of ${\rm At}(E_H)(-\log D)$ and $TM(-\log D)$, and hence we have
$\beta({\mathbb K}(v_1,\, v_2))\,=\, 0$. Consequently, from \eqref{e8} it follows that
${\mathbb K}(v_1,\, v_2)$ is a holomorphic section
of $\text{ad}(E_H)$ over $U$. From the identity $[fv,\, w]\,=\, f[v,\, w]- w(f)\cdot v$, where
$f$ is a holomorphic function while $v$ and $w$ are holomorphic vector fields, it follows that
$$
{\mathbb K}(fv_1,\, v_2) \,=\, f {\mathbb K}(v_1,\, v_2)\, .
$$
Also, we have ${\mathbb K}(v_1,\, v_2)\,=\, -{\mathbb K}(v_2,\, v_1)$. Therefore, the mapping
$(v_1,\, v_2)\, \longmapsto\, {\mathbb K}(v_1,\, v_2)$ defines a holomorphic section
\begin{equation}\label{e13}
{\mathbb K}(\varphi)\, \in\, H^0(M,\, \Omega^2_M(\log D)\otimes {\rm ad}(E_H))\, .
\end{equation}
The section ${\mathbb K}(\varphi)$ in \eqref{e13} is called the \textit{curvature} of the
logarithmic connection $\varphi$.

\subsection{Residue}

Restricting
to $D$ the exact sequences in \eqref{e8} and \eqref{f3}, we get the following commutative diagram
\begin{equation}\label{e9}
\begin{matrix}
0& \longrightarrow & \text{ad}(E_H)\vert_D & \stackrel{\widehat{\iota}_0}{\longrightarrow} &
{\rm At}(E_H)(-\log D)\vert_D & \stackrel{\widehat{\beta}}{\longrightarrow} &
TM(-\log D)\vert_D & \longrightarrow\,0 &\\
&& \Vert && ~\Big\downarrow\mu && ~\Big\downarrow\psi\\
0& \longrightarrow & \text{ad}(E_H)\vert_D & \stackrel{\iota_1}{\longrightarrow} &
{\rm At}(E_H)\vert_D & \stackrel{\widehat{\beta}'}{\longrightarrow} &
TM\vert_D & \longrightarrow\,0 &
\end{matrix}
\end{equation}
whose rows are exact; the map $\psi$ is the one in \eqref{fn} and $\mu$ is the homomorphism given 
by the natural homomorphism ${\rm At}(E_H)(-\log D) \,\longrightarrow\, {\rm At}(E_H)$. In 
\eqref{e9} the following convention is employed: the restriction to $D$ of a map on $M$ is 
denoted by the same symbol after adding a hat. From \eqref{f1} we know that the kernel of $\psi$ 
is $\mathbb L\,=\, \nu_*{\mathcal O}_{\widetilde{D}}$ (see \eqref{f0}). Let
$$
\iota_{\mathbb L}\, :\, {\mathbb L}\, \longrightarrow\, TM(-\log D)\vert_D
$$
be the inclusion map.

Let $\varphi\, :\, TM(-\log D)\, \longrightarrow\, {\rm At}(E_H)(-\log D)$ be a
logarithmic connection on $E_H$ singular over $D$. Consider the composition
$$
\widehat{\varphi}\circ\iota_{\mathbb L}\,:\, {\mathbb L}\, \longrightarrow\,
{\rm At}(E_H)(-\log D)\vert_D
$$
(the restriction of $\varphi$ to $D$ is denoted by $\widehat{\varphi}$). From the
commutativity of the diagram in \eqref{e9} it follows that
\begin{equation}\label{e11}
\widehat{\beta}'\circ\mu\circ \widehat{\varphi}\circ\iota_{\mathbb L}\,=\,
\psi\circ \widehat{\beta}\circ\widehat{\varphi}\circ\iota_{\mathbb L}\,.
\end{equation}
But $\widehat{\beta}\circ\widehat{\varphi}\,=\, \text{Id}_{TM(-\log D)\vert_D}$
by \eqref{e10}, while $\psi\circ\iota_{\mathbb L}\,=\, 0$ by \eqref{f1}, so these two together
imply that $\psi\circ \widehat{\beta}\circ\widehat{\varphi}\circ\iota_{\mathbb L}\,=\, 0$.
Hence from \eqref{e11} we conclude that
$$
\widehat{\beta}'\circ\mu\circ \widehat{\varphi}\circ\iota_{\mathbb L}\,=\,0\, .
$$
Now from the exactness of the bottom row in \eqref{e9} it follows that the image of
$\mu\circ \widehat{\varphi}\circ\iota_{\mathbb L}$ is contained in the image of the
injective map $\iota_1$ in \eqref{e9}. Therefore,
$\mu\circ \widehat{\varphi}\circ\iota_{\mathbb L}$ defines a map
\begin{equation}\label{e12}
{\mathcal R}_\varphi\, :\, {\mathbb L}\, \longrightarrow\, \text{ad}(E_H)\vert_D\, .
\end{equation}
The homomorphism ${\mathcal R}_\varphi$ in \eqref{e12} is called the residue of
the logarithmic connection $\varphi$ \cite{De}.

\section{Logarithmic Cartan geometry}\label{sect3}

\subsection{Definition}

Let $G$ be a complex connected Lie group and $H\, \subset\, G$ a complex Lie subgroup.
The Lie algebras of $G$ and $H$ will be denoted by $\mathfrak g$ and $\mathfrak h$
respectively. We recall that a holomorphic Cartan geometry of type $(G,\, H)$ on a complex
manifold $M$ is a pair of the form $(E'_H,\, \theta')$, where $E'_H$ is a holomorphic
principal $H$--bundle over $M$, and
$$
\theta'\, :\, TE'_H\, \longrightarrow\, E'_H\times {\mathfrak g}
$$
is a holomorphic homomorphism of vector bundles over $E'_H$ such that
\begin{enumerate}
\item $\theta'$ is an isomorphism,

\item $\theta'$ is $H$--equivariant (the action of $H$ on $TE'_H$ is given by the
action of $H$ on $E'_H$, while the action of $H$ on $\mathfrak g$ is given by conjugation), and

\item the restriction of $\theta'$ to the fiber $(E'_H)_m$ coincides with the
Maurer--Cartan form of $H$ for every point $m \, \in\, M$.
\end{enumerate}
(see \cite{Sh} for more details).

Let $E'_G\,:=\, E'_H(G) \, =\, E'_H\times^H G$ be the holomorphic principal $G$--bundle
on $M$ obtained by extending the structure group of the holomorphic principal $H$--bundle
$E'_H$ using the inclusion of $H$ in $G$. The adjoint bundle of $E'_G$ will be
denoted by $\text{ad}(E'_G)$. The inclusion of $\mathfrak h$ in $\mathfrak g$ produces an
injective homomorphism of holomorphic Lie algebra bundles
$$
\text{ad}(E'_H)\, \longrightarrow\, \text{ad}(E'_G)\, .
$$
Giving a homomorphism $\theta'$ satisfying the above 
three conditions is equivalent to giving a holomorphic isomorphism
$$
\theta''\, :\, \text{At}(E'_H) \, \longrightarrow\, \text{ad}(E'_G)\, ,
$$
where $\text{At}(E'_H)$ is the Atiyah bundle for $E'_H$, such that the following diagram
is commutative
\begin{equation}\label{cd2}
\begin{matrix}
0 &\longrightarrow & {\rm ad}(E'_H) & \longrightarrow &
{\rm At}(E'_H) & \longrightarrow & TM &\longrightarrow & 0\\
&& \Vert && ~\,~\,~\, \Big\downarrow \theta'' && \Big\downarrow\\
0 &\longrightarrow & {\rm ad}(E'_H) &\longrightarrow & {\rm ad}(E'_G) &\longrightarrow &
{\rm ad}(E'_G)/{\rm ad}(E'_H)&\longrightarrow & 0
\end{matrix}
\end{equation}
with the top row being the Atiyah exact sequence for $E'_H$ (see \eqref{f3}) (see, for example, 
\cite{BD1, BD2}).

Fix a pair $(V,\, \chi)$, where $V$ is a finite dimensional complex vector space, and
$$
\chi\, :\, G\, \longrightarrow\, \text{GL}(V)
$$
is a holomorphic homomorphism satisfying the condition that the corresponding homomorphism
of Lie algebras
\begin{equation}\label{dc}
d\chi\, :\, {\mathfrak g}\, \longrightarrow\, \text{Lie}(\text{GL}(V))\,=\, \text{End}(V)
\end{equation}
is injective. Notice that such a homomorphism always exists for $G$ simply connected (by Ado's 
Theorem) and for $G$ semi-simple (see Theorem 3.2, Chapter XVII in \cite{Ho}). Complex Lie groups 
$G$ admitting holomorphic linear representations with discrete kernel are called {\it complex 
affine}. A complex Lie group with finitely many connected components is complex affine exactly 
when it admits a holomorphic finite dimensional faithful representation, which occurs just when 
its identity component is a holomorphic semidirect product of a connected and simply connected 
solvable complex Lie group and a connected reductive complex linear algebraic group \cite[p. 601, 
Theorem 16.3.7]{Hilg}.

Let $E'_H$ be a holomorphic principal $H$--bundle over $M$ and
$$
\theta'\, :\, TE'_H\, \longrightarrow\, E'_H\times {\mathfrak g}
$$
a holomorphic homomorphism of vector bundles such that
\begin{enumerate}
\item $\theta'$ is $H$--equivariant, and

\item the restriction of $\theta'$ to the fiber $(E'_H)_m$ coincides with the
Maurer--Cartan form of $H$ for every point $m\, \in\, M$.
\end{enumerate}
The holomorphic principal $\text{GL}(V)$--bundle
$$E'_H(\text{GL}(V))\,=\, E'_H\times^\chi \text{GL}(V)$$ over $M$, obtained by extending
the structure group of $E'_H$ using the homomorphism $\chi\vert_H$, will be denoted
by $E'_H(V)$.

\begin{lemma}\label{lem0}
The above homomorphism $\theta'$ produces a holomorphic ${\rm End}(V)$--valued $1$--form on
the total space of $E'_H(V)$ that defines a holomorphic connection on the holomorphic
principal ${\rm GL}(V)$--bundle $E'_H(V)$.
\end{lemma}

\begin{proof}
We recall that $E'_H(V)$ is a quotient of $E'_H\times \text{GL}(V)$ where two points
$(y_1,\, g_1),\, (y_2,\, g_2)\,\in\, E'_H\times \text{GL}(V)$ are identified if there is an element
$h\, \in\, H$ such that $y_2\,=\, y_1h$ and $g_2\,=\, \chi(h)^{-1}g_1$. Now $\theta'$ and
the Maurer--Cartan form on $\text{GL}(V)$ (for the left--translation action of
$\text{GL}(V)$ on itself) together define a holomorphic $1$--form $\omega$
on $E'_H\times \text{GL}(V)$ with values in the Lie algebra $\text{End}(V)$. More precisely,
for tangent vectors $v\, \in\, T_y E'_H$ and $w\, \in\, T_g \text{GL}(V)$,
$$
\omega_{(y, g)}(v,\, w)\,=\, (d \chi \circ Ad(g^{-1}) \circ \theta')(v) +MC_g(w)\, ,
$$
where $MC$ denotes the Maurer--Cartan form on $\text{GL}(V)$ for the left--translation action of 
$\text{GL}(V)$ on itself and $Ad$ denotes the adjoint representation of $G$ in its Lie algebra, 
while $d\chi$ is the homomorphism in \eqref{dc}. Now it is straight-forward to check that 
$\omega$ is $H$-invariant and vanishes on the $H$-orbits. It follows that $\omega$ is basic: it 
descends to the quotient space $E'_H(V)$ as a holomorphic $1$--form with values in 
$\text{End}(V)$. This $\text{End}(V)$--valued $1$--form on $E'_H(V)$ clearly defines a 
holomorphic connection.

To describe the above connection on $E'_H(V)$ as a splitting of the Atiyah exact sequence,
we first note that the Atiyah bundle $\text{At}(E'_H(V))$ (see \eqref{f3}) is the quotient
$$
\text{At}(E'_H(V))\,=\, (\text{At}(E'_H)\oplus \text{ad}(E'_H(V)))/\text{ad}(E'_H)
$$
for the homomorphism
\begin{equation}\label{ab}
\xi\, :\, \text{ad}(E'_H)\, \longrightarrow\, \text{At}(E'_H)\oplus
\text{ad}(E'_H(V))
\end{equation}
which is constructed as follows. Since $E'_H(V)$ is the principal
$\text{GL}(V)$--bundle on $M$ obtained by extending the structure group of the $E'_H$
using $\chi\vert_H$, the corresponding homomorphism of Lie algebras
$$
d\chi\vert_H\, :\, {\mathfrak H}\, \longrightarrow\, \text{End}(V)
$$
produce a homomorphism
$$
\alpha\, :\, \text{ad}(E'_H)\, \longrightarrow\, \text{ad}(E'_H(V))\, .
$$
Let $\iota_0$ be the inclusion of $\text{ad}(E'_H)$ in $\text{At}(E'_H)$ (see \eqref{f3}).
The homomorphism $\xi$ in \eqref{ab} is defined by $v\,\longmapsto \, (\iota_0(v),\, - \alpha(v))$.

As noted before, the homomorphism $\theta'$ produces a homomorphism
$$\theta''\, :\, \text{At}(E'_H)\, \longrightarrow\, \text{ad}(E'_G)\, ;$$
the homomorphism $\theta''$ has the property that the diagram in \eqref{cd2}
is commutative. Since $E'_H(V)$ coincides with the principal
$\text{GL}(V)$--bundle on $M$ obtained by extending the structure group of the principal
$G$--bundle $E'_G$
using $\chi$, the homomorphism of Lie algebras $d\chi\, :\, {\mathfrak g}\, \longrightarrow\,
\text{End}(V)$ produces a holomorphic homomorphism Lie algebra bundles
$$\alpha'\, :\, \text{ad}(E'_G)\, \longrightarrow\, \text{ad}(E'_H(V))\, .$$

Now consider the homomorphism
$$
\widehat{\varphi}'\,:\, \text{At}(E'_H)\oplus \text{ad}(E'_H(V))\, \longrightarrow\,
\text{ad}(E'_H(V))\, ,\ \ (v,\, w)\, \longmapsto\, \alpha'\circ \theta''(v)+w\, .
$$
Since $\widehat{\varphi}'$ vanishes on the image of the homomorphism $\xi$ in \eqref{ab}, we
conclude that $\widehat{\varphi}'$ descends to a homomorphism
$$
{\varphi}'\,:\, \text{At}(E'_H(V))\, \longrightarrow\, \text{ad}(E'_H(V))
$$
from the quotient bundle $\text{At}(E'_H(V))/\xi(\text{ad}(E'_H))\,=\, \text{At}(E'_H(V))$.
It is straightforward to check that $\varphi'$ gives a holomorphic splitting of the
Atiyah exact sequence for $E'_H(V)$. Therefore, $\varphi'$ defines a holomorphic connection
on $E'_H(V)$.
\end{proof}

For notational convenience the quadruple $(H,\, G, \, V,\, \chi)$ will be denoted
by $\mathbb H$.

As before, $D\, \subset\, M$ is a normal crossing divisor.

\begin{definition}\label{def1}
A \textit{logarithmic Cartan geometry of type} $\mathbb H$ on $M$ with polar part on
$D$ is a triple of the form $(E_H,\, \theta,\, \widehat{E}_H(V))$, where
\begin{itemize}
\item $E_H$ is a holomorphic
principal $H$--bundle over the complement $M\setminus D$, and
$$
\theta\, :\, TE_H\, \longrightarrow\, E_H\times {\mathfrak g}
$$
is a holomorphic homomorphism of vector bundles over $E_H$, such that
$(E_H,\, \theta)$ is a holomorphic Cartan geometry of type $(G,\, H)$ on
$M\setminus D$, and

\item $q_0\, :\, \widehat{E}_H(V)\, \longrightarrow\, M$ is an extension of
the principal ${\rm GL}(V)$--bundle $E_H(V)$ on $M\setminus D$ to a
holomorphic principal ${\rm GL}(V)$--bundle on $M$ such that the homomorphism
$$
TE_H(V)\, \longrightarrow\, E_H(V)\times {\rm End}(V)\, ,
$$
constructed in Lemma \ref{lem0} from $\theta$, extends to a homomorphism
$$
T\widehat{E}_H(V)(-\log q^{-1}_0(D)) \, \longrightarrow\, \widehat{E}_H(V)\times {\rm End}(V)
$$
(note that $q^{-1}_0(D)\, \subset\, \widehat{E}_H(V)$ is a normal crossing divisor).
\end{itemize}
\end{definition}

Consider the holomorphic vector bundle $E_H(V)\times^{{\rm GL}(V)}V$ on $M\setminus D$ 
associated to the holomorphic principal ${\rm GL}(V)$--bundle $E_H(V)$ for the standard 
action of ${\rm GL}(V)$ on $V$. For notational convenience, this vector bundle on $M\setminus D$ will
be denoted by $E^V_H$. We note that any connection on the principal ${\rm GL}(V)$--bundle $E_H(V)$ 
induces a connection on the associated vector bundle $E^V_H$. Conversely, any
connection on $E^V_H$ produces a connection on the principal ${\rm GL}(V)$--bundle $E_H(V)$.
More precisely, there is a natural bijection between the connections on $E^V_H$ and
the connections on the principal ${\rm GL}(V)$--bundle $E_H(V)$.

The following lemma produces an alternative formulation of the definition of
a logarithmic Cartan geometry of type $\mathbb H$ on $M$ with polar part on
$D$.

\begin{lemma}\label{lem-1}
Take a pair $(E_H,\, \theta)$ defining a holomorphic Cartan geometry of type $(G,\, H)$ on 
$M\setminus D$. Giving an extension $q_0\, :\, \widehat{E}_H(V)\, \longrightarrow\, M$ of 
$E_H(V)$ to a holomorphic principal ${\rm GL}(V)$--bundle on $M$, such that $(E_H,\, 
\theta,\, \widehat{E}_H(V))$ is a logarithmic Cartan geometry of type $\mathbb H$ on $M$, 
is equivalent to giving an extension of the holomorphic vector bundle $E^V_H$ on 
$M\setminus D$ to a holomorphic vector bundle $\widehat{E}^V_H$ on $M$ such that
holomorphic connection on $E^V_H$ given by $\theta$ in Lemma \ref{lem0} extends to
a logarithmic connection on the holomorphic vector bundle $\widehat{E}^V_H$.
\end{lemma}

\begin{proof}
This is a consequence of the following general fact. Let $F$ be a holomorphic vector
bundle on $M$ whose rank coincides with the dimension of $V$. Let
$$
q_1\, :\, {\mathbb F}\, \longrightarrow\, M
$$
denote the associated holomorphic principal ${\rm GL}(V)$--bundle on $M$; so $\mathbb F$ is
the space of all isomorphisms from $V$ to the fibers of $F$. Let $\nabla$ be a holomorphic 
connection on the restriction $F\vert_{M\setminus D}$. The ${\rm End}(V)$--valued holomorphic 
$1$--form on ${\mathbb F}\vert_{M\setminus D}$ giving the connection on ${\mathbb 
F}\vert_{M\setminus D}$ corresponding to $\nabla$ will be denoted by $\omega_{\nabla}$. Then 
$\nabla$ is a logarithmic connection on $F$ if and only if $\omega_{\nabla}$ extends to
a homomorphism
$$
T{\mathbb F}(-\log q^{-1}_1(D))\, \longrightarrow\, {\mathbb F}\times {\rm End}(V)
$$
over $\mathbb F$. The lemma follows immediately from this.
\end{proof}

\subsection{Flatness}

A logarithmic Cartan geometry $(E_H,\, \theta,\, \widehat{E}_H(V))$, of type $\mathbb H$ on 
$M$ with polar part on $D$, is called \textit{flat} if the curvature of the logarithmic
connection on $\widehat{E}^V_H$ given by $\theta$ vanishes identically. Clearly,
the curvature of the logarithmic
connection on $\widehat{E}^V_H$ given by $\theta$ vanishes identically if and only if the
curvature of the holomorphic connection on $E^V_H$ given by $\theta$ vanishes identically.

\subsection{A construction of logarithmic Cartan geometry}

As before, $M$ is a connected complex manifold with a normal crossing divisor $D$.
Let $N$ be a connected complex manifold, and let
\begin{equation}\label{vp}
\varpi\,:\, N\,\longrightarrow \, M
\end{equation}
be a ramified finite Galois covering such that the ramification locus in $M$ coincides
with $D$. The Galois group for $\varpi$ will be denoted by $\Gamma$.

Take ${\mathbb H}\, =\, (H,\, G, \, V,\, \chi)$
as above. Let $p'\, :\, E'_H\, \longrightarrow\, N$ be a holomorphic principal $H$--bundle
on $N$ equipped with an action of $\Gamma$
$$
\rho\, :\, \Gamma\times E'_H\, \longrightarrow\, E'_H
$$
satisfying the following conditions:
\begin{itemize}
\item the projection $p'$ is $\Gamma$--equivariant,

\item the actions of $\Gamma$ and $H$ on $E'_H$ commute, and

\item for every $g\, \in\, \Gamma$, the diffeomorphism
$E'_H\, \longrightarrow\, E'_H$ defined by $z\, \longmapsto\, \rho(g,\, z)$ is
holomorphic.
\end{itemize}
The action of $\Gamma$ on $E'_H$ produces an action of $\Gamma$ on $TE'_H$. Let
$$
\theta'\, :\, TE'_H\, \longrightarrow\, E'_H\times{\mathfrak g}
$$
be a holomorphic isomorphism of vector bundles such that
\begin{itemize}
\item the pair $(E'_H,\, \theta')$ defines a holomorphic Cartan geometry of type
$(G,\, H)$ on $N$, and

\item the homomorphism $\theta'$ is $\Gamma$--equivariant.
\end{itemize}

\begin{theorem}\label{thm1}
The above pair $(E'_H,\, \theta')$ produces a logarithmic Cartan geometry of type $\mathbb H$
on $M$ with polar part on $D$.
\end{theorem}

\begin{proof}
Consider $\varpi$ in \eqref{vp}. Since the restriction
$$
\varpi\vert_{N\setminus\varpi^{-1}(D)}\, :\, N\setminus\varpi^{-1}(D)\,\longrightarrow\,
M\setminus D
$$
is an \'etale Galois covering, the quotient
\begin{equation}\label{ehq}
E_H\, :=\, (E'_H\vert_{N\setminus\varpi^{-1}(D)})/\Gamma
\end{equation}
is a holomorphic principal $H$--bundle
on $M\setminus D$. The homomorphism $\theta'$, being $\Gamma$--equivariant, descends to
a homomorphism
\begin{equation}\label{thq}
\theta\,:\, TE_H\, \longrightarrow\, E_H\times{\mathfrak g}\, .
\end{equation}
It is evident that this pair $(E_H,\, \theta)$ defines a holomorphic Cartan geometry of type
$(G,\, H)$ on $M\setminus D$.

Let $${\mathbb W}\, :=\, E'_H\times^\chi V\, \longrightarrow\, N$$ be the holomorphic vector
bundle over $N$ associated to the principal $H$--bundle $E'_H$ for the action of $H$ on $V$ given by
$\chi\vert_H$.
Note that the holomorphic principal $\text{GL}(V)$--bundle $E'_H(V)$, obtained by extending the
structure group of the principal $H$--bundle $E'_H$ using $\chi\vert_H$, coincides with the
frame bundle for $\mathbb W$ (this frame bundle is the space of all isomorphisms from $V$ to the
fibers of $\mathbb W$).
The action of $\Gamma$ on $E'_H$ induces an action of $\Gamma$
on every fiber bundle associated to $E'_H$. In particular, $\Gamma$ acts on
the vector bundle ${\mathbb W}$. More
explicitly, the action of $\Gamma$ on $E'_H$ and the trivial action of $\Gamma$ on $V$ together
produce an action of $\Gamma$ on $E'_H\times V$. This action of $\Gamma$ on $E'_H\times V$ descends
to the quotient space $\mathbb W$ of $E'_H\times V$.

Consider the direct image $\varpi_*{\mathbb W}$ on $M$, where $\varpi$ is the map in \eqref{vp}.
It is a locally free coherent
analytic sheaf, because $\varpi$ is a finite map (higher direct images vanish). In other words,
$\varpi_*{\mathbb W}$
is a holomorphic vector bundle on $M$.
The action of $\Gamma$ on ${\mathbb W}$ produces an action of $\Gamma$ on
the holomorphic vector bundle $\varpi_*{\mathbb W}$. For any $g\,\in\, \Gamma$, let
\begin{equation}\label{tg}
\tau_g\, :\, \varpi_*{\mathbb W}\,\longrightarrow\, \varpi_*{\mathbb W}
\end{equation}
be the automorphism of $\varpi_*{\mathbb W}$ given by this action $g$ on it.
Consider the coherent analytic sheaf on $M$ given by the $\Gamma$--invariant part
$$
(\varpi_*{\mathbb W})^\Gamma\,\subset\, \varpi_*{\mathbb W}\, .
$$
Since $\Gamma$ is a finite group, the inclusion of $(\varpi_*{\mathbb W})^\Gamma$ in
$\varpi_*{\mathbb W}$ splits holomorphically. In fact the kernel of the endomorphism
$$
\sum_{g\in \Gamma} \tau_g\, :\, \varpi_*{\mathbb W}\,\longrightarrow\, \varpi_*{\mathbb W}
$$
(the homomorphism $\tau_g$ is defined in \eqref{tg})
is a direct summand of the $\Gamma$--invariant part $(\varpi_*{\mathbb W})^\Gamma$,
while the image of $\sum_{g\in \Gamma} \tau_g$ coincides with $(\varpi_*{\mathbb W})^\Gamma$. Since
$(\varpi_*{\mathbb W})^\Gamma$ is a direct summand of the holomorphic vector bundle
$\varpi_*{\mathbb W}$, we conclude that the coherent analytic sheaf
\begin{equation}\label{ww}
\widehat{\mathbb W}\,:=\, (\varpi_*{\mathbb W})^\Gamma
\end{equation}
is also a holomorphic vector bundle on $M$.

The restriction $\widehat{\mathbb W}\vert_{M\setminus D}$ is clearly identified with
the quotient $({\mathbb W}\vert_{N\setminus\varpi^{-1}(D)})/\Gamma$, and hence
$\widehat{\mathbb W}\vert_{M\setminus D}$ is the
holomorphic vector bundle over $M\setminus D$ associated to the principal $H$--bundle
$E_H$ in \eqref{ehq} for the action of $H$ on $V$ given by $\chi\vert_H$.

From Lemma \ref{lem0} we know that $\theta$ produces a holomorphic connection on the holomorphic 
vector bundle over $M\setminus D$ associated to the principal $H$--bundle $E_H$ for the action of 
$H$ on $V$ given by $\chi\vert_H$. In view of the above mentioned isomorphism of this vector 
bundle with $\widehat{\mathbb W}\vert_{M\setminus D}$, we conclude that $\theta$ produces a 
holomorphic connection on $\widehat{\mathbb W}\vert_{M\setminus D}$. Let $\nabla^V$ denote this 
holomorphic connection on $\widehat{\mathbb W}\vert_{M\setminus D}$ given by $\theta$.

We want to show that the triple $(E_H,\, \theta,\, \widehat{\mathbb W})$ defines a logarithmic 
Cartan geometry of type $\mathbb H$ on $M$ with polar part on $D$. In view of Lemma \ref{lem-1}, 
it suffices to prove that the above holomorphic connection $\nabla^V$ on $\widehat{\mathbb 
W}\vert_{M\setminus D}$ is a logarithmic connection on $\widehat{\mathbb W}$.

Let
$$
\nabla^{\mathbb W}\, :\, {\mathbb W}\, \longrightarrow\, {\mathbb W}\otimes \Omega^1_N
$$
be the holomorphic connection on ${\mathbb W}$ constructed using $\theta'$ in Lemma
\ref{lem0}. It gives a homomorphism of sheaves
$$
\varpi_*\nabla^{\mathbb W}\, :\, \varpi_*{\mathbb W}\, \longrightarrow\,
\varpi_*({\mathbb W}\otimes \Omega^1_N)\, .
$$
Since $\theta'$ is $\Gamma$--equivariant, it follows that this homomorphism
$\varpi_*\nabla^{\mathbb W}$ maps the invariant subsheaf $(\varpi_*{\mathbb W})^\Gamma$
to $(\varpi_*({\mathbb W}\otimes \Omega^1_N))^\Gamma$. Let
$$
(\varpi_*\nabla^{\mathbb W})^\Gamma\, :\, \widehat{\mathbb W}\,:=\,(\varpi_*{\mathbb W})^\Gamma
\, \longrightarrow\, (\varpi_*({\mathbb W}\otimes \Omega^1_N))^\Gamma
$$
be this restriction of $\varpi_*\nabla^{\mathbb W}$. We know that
$$
(\varpi_*({\mathbb W}\otimes \Omega^1_N))^\Gamma\, \subset\, (\varpi_*{\mathbb W})^\Gamma\otimes
\Omega^1_M(\log D)\,=\, \widehat{\mathbb W}\otimes \Omega^1_M(\log D)
$$
\cite[p.~525, Lemma 4.11]{Bi}. Consequently, the above homomorphism $(\varpi_*\nabla^{\mathbb
W})^\Gamma$ defines a logarithmic connection on the holomorphic
vector bundle $\widehat{\mathbb W}$.

On the other hand, the restriction of this logarithmic connection $(\varpi_*\nabla^{\mathbb 
W})^\Gamma$ to $M\setminus D$ clearly coincides with the connection $\nabla^V$ on 
$\widehat{\mathbb W}\vert_{M\setminus D}$ constructed earlier from $\theta$. Therefore, we 
conclude that the above holomorphic connection $\nabla^V$ on the holomorphic vector bundle 
$\widehat{\mathbb W}\vert_{M\setminus D}$ is a logarithmic connection on $\widehat{\mathbb 
W}$. As noted before, this completes the proof of the theorem.
\end{proof}

\section{Logarithmic affine structure}\label{sect4}

In this section we study logarithmic Cartan geometries modelled on the complex affine space.

Consider the semidirect product
$$
G\,:=\, {\mathbb C}^d\rtimes\text{GL}(d, {\mathbb C})$$ for the
standard action of $\text{GL}(d, {\mathbb C})$ on ${\mathbb C}^d$. Note that
$G$ is the group of affine transformations of ${\mathbb C}^d$. Also,
$G$ is realized as a closed algebraic subgroup of $\text{GL}(d+1, {\mathbb C})$
in the following way. Consider all linear automorphisms $A$ of
${\mathbb C}^d\oplus {\mathbb C}\,=\, {\mathbb C}^{d+1}$ such that
$A({\mathbb C}^d)\,=\, {\mathbb C}^d$ and $A(0,\, 1)\,=\, (v,\, 1)$, where
$0,\, v\, \in\, {\mathbb C}^d$. The element $(v,\, A\vert_{{\mathbb C}^d})\, \in\, G$
is mapped to $A\,\in\, \text{GL}(d+1, {\mathbb C})$.

Let $H\,:=\, \text{GL}(d, {\mathbb C})\, \subset\, G$ be the complex algebraic;
it is the isotropy subgroup for $0\, \in\, {\mathbb C}^d$ for the action of $G$
on ${\mathbb C}^d$. Set $V\,=\, {\mathbb C}^{d+1}$, and
$$\chi\, :\, G\, \longrightarrow\, \text{GL}(V)$$ to be the restriction to
$G$ of the standard action of $\text{GL}(d+1, {\mathbb C})$ on ${\mathbb C}^{d+1}$.
As before, denote $(H,\, G, \, V,\, \chi)$ by $\mathbb H$.

A holomorphic affine structure on a complex manifold $M$ is a holomorphic Cartan geometry
on $M$ of type $(G,\, H)$. Let $(E_H,\, \theta)$ be a holomorphic Cartan geometry
of type $(G,\, H)$ on $M$. As before, the holomorphic vector bundle $E_H\times^\chi V$ associated
to $E_H$ for the homomorphism $\chi\vert_H$ will be denote by $E^V_H$. The homomorphism
$\theta$ produces a holomorphic connection on the holomorphic vector bundle $E^V_H$
(see Lemma \ref{lem0}). The holomorphic connection on $E^V_H$ induced by $\theta$ will be
denoted by $\nabla^V$.

As before, $E_G$ denotes the holomorphic principal $G$--bundle on $M$ obtained by extending 
the structure group of $E_H$ using the inclusion of $H$ in $G$. It is known that $\theta$ 
induces a holomorphic connection on the principal $G$--bundle $E_G$ (see, for example,
Appendix A, Section 3 in \cite{Sh}). Let $\nabla^G$ 
denote this holomorphic connection on $E_G$ given by $\theta$. We note that the holomorphic 
connection on the associated holomorphic vector bundle $E_G\times^\chi V\,=\, E^V_H$ induced
by $\nabla^G$ coincides with the above connection $\nabla^V$. Let
\begin{equation}\label{phi}
\phi\, :\, G\, \longrightarrow\, \text{GL}(d, {\mathbb C})
\end{equation}
be the natural projection.
The holomorphic vector bundle $E_G\times^\phi {\mathbb C}^d$ of rank $d$ on $M$,
associated to $E_G$ for the action of $G$ on ${\mathbb C}^d$ given by $\phi$ in \eqref{phi}, is the
holomorphic tangent bundle $TM$. The connection $\nabla^G$ on $E_G$ induces a holomorphic
connection on the associated vector bundle $E_G\times^\phi {\mathbb C}^d\,=\, TM$. This
connection on $TM$ will be denoted by $\nabla^T$.

As before, $D\, \subset\, M$ is a normal crossing divisor.

A logarithmic affine structure on $M$ with polar part on $D$ is a logarithmic
Cartan geometry on $M$ of type $\mathbb H$ with polar part on $D$. Therefore,
a logarithmic affine structure on $M$ with polar part on $D$ consists of
\begin{itemize}
\item a holomorphic Cartan geometry $(E_H,\, \theta)$ on $M\setminus D$ of type $(G,\, H)$, and

\item a holomorphic extension of the holomorphic vector bundle $E^V_H$ on
$M\setminus D$ to a holomorphic vector bundle $\widetilde{E}^V_H$ on $M$ such that
the holomorphic connection on $E^V_H$ given by $\theta$ is a logarithmic connection
on the holomorphic vector bundle $\widetilde{E}^V_H$.
\end{itemize}

The complement $M\setminus D$ will be denoted by $M'$.
Let $(E_H,\, \theta,\, \widetilde{E}^V_H)$ be a logarithmic affine structure on $M$
with polar part on $D$. Consider the holomorphic connection on $TM'$ given by $\theta$.
Since $\nabla^V$ is a logarithmic connection on $\widetilde{E}^V_H$, and $TM'$ is a
holomorphic subbundle of $E^V_H$ preserved by $\nabla^V$, it follows that
$TM'$ generated a holomorphic subbundle $\widetilde{T}\, \subset\, \widetilde{E}^V_H$
such that
\begin{enumerate}
\item $\widetilde{T}\vert_{M'}\,=\, TM'\, \subset\, E^V_H$,

\item $\widetilde{T}$ is preserved by the logarithmic connection on $\widetilde{E}^V_H$
given by $\theta$, and

\item the restriction to the subbundle $\widetilde{T}$, of the logarithmic connection on
$\widetilde{E}^V_H$, is a logarithmic connection.
\end{enumerate}

Consider the standard action of
$$
H\, \subset\, G\, \subset\,
{\mathbb C}^d\rtimes\text{GL}(d, {\mathbb C})\, \subset\, \text{GL}(d+1, {\mathbb C})
$$
on $V\,=\,{\mathbb C}^{d+1}$. This $H$--module decomposes as
\begin{equation}\label{dm}
{\mathbb C}^{d+1}\,=\, {\mathbb C}^{d}\oplus {\mathbb C}\, ,
\end{equation}
where the action of $H\,=\, \text{GL}(d, {\mathbb C})$ on ${\mathbb C}^{d}$ it the
standard one and the action of $H$ on $\mathbb C$ is the trivial one.

Let $(E_H,\, \theta)$ be a holomorphic affine structure on $M'\,:=\, M\setminus D$.
Using the decomposition of the $H$--module in \eqref{dm}, the holomorphic vector bundle
$E^V_H$ on $M\setminus D$ holomorphically decomposes as
\begin{equation}\label{ed2}
E^V_H\,=\, TM'\oplus {\mathcal O}_{M'}\, .
\end{equation}

Let $\widetilde{TM'}\, \longrightarrow\, M$ be a holomorphic vector bundle on $M$ that extends
$TM'$, meaning $\widetilde{TM'}\vert_{M'}\,=\, TM'$. Then using \eqref{ed2} it follows that
\begin{equation}\label{ed}
\widetilde{E}^V_H\,=\, \widetilde{TM'}\oplus {\mathcal O}_{M}
\end{equation}
is an extension of ${E}^V_H$ to a holomorphic vector bundle over $M$.

As before, let $\nabla^V$ be the holomorphic connections on $E^V_H$
given by $\theta$, and let $\nabla^T$ denote the holomorphic connections on $TM'$
given by the holomorphic connection on $E_G$ (given by $\theta$) and the homomorphism
$\phi$ in \eqref{phi}.

\begin{proposition}\label{prop2}
If the holomorphic connection $\nabla^V$ is a logarithmic connection on the holomorphic
vector bundle $\widetilde{E}^V_H$ in \eqref{ed}, then the holomorphic connection
$\nabla^T$ on $TM'$ is a logarithmic connection on $\widetilde{TM'}$.
\end{proposition}

\begin{proof}
Consider the holomorphic subbundle ${\mathcal O}_{M'}\, \subset\, E^V_H$ in \eqref{ed2}.
The holomorphic connection $\nabla^V$ on $E^V_H$ preserves this subbundle. Hence
$\nabla^V$ induces a holomorphic connection on the quotient bundle
$E^V_H/{\mathcal O}_{M'}\,=\, TM'$. This induced connection on $TM'$ coincides with the
holomorphic connection $\nabla^T$ on $TM'$. From this it follows immediately that
if the holomorphic connection $\nabla^V$ is a logarithmic connection on the holomorphic
vector bundle $\widetilde{E}^V_H$ in \eqref{ed}, then the holomorphic connection
$\nabla^T$ on $TM'$ is a logarithmic connection on $\widetilde{TM}'$.
\end{proof}

The converse of Proposition \ref{prop2} it not true in general, meaning we can have
a situation where the holomorphic connection
$\nabla^T$ on $TM'$ is a logarithmic connection on $\widetilde{TM'}$, but the
holomorphic connection $\nabla^V$ is not a logarithmic connection on the holomorphic
vector bundle $\widetilde{E}^V_H$. However, the following is straightforward to prove.

\begin{proposition}\label{prop3}
Assume that
\begin{itemize}
\item the holomorphic connection
$\nabla^T$ on $TM'$ is a logarithmic connection on $\widetilde{TM'}$, and

\item the second fundamental form of the subbundle $TM'\, \subset\, E^V_H$ in \eqref{ed2}
extends to a section of $\Omega^2_M(\log D)$.
\end{itemize}
Then the holomorphic connection $\nabla^V$ is a logarithmic connection on the holomorphic
vector bundle $\widetilde{E}^V_H$ in \eqref{ed}.
\end{proposition}

It may be mentioned that obstructions for a compact complex manifold to admit logarithmic
affine and projective structures were found in \cite{Ka}.

%%%%%%%%%%%%%%%%%%%%%%%%%%%%%%%%%%%%%%%%%%%%%%%%%%%%%%%%%%%%%%

\end{document}